\documentclass{amsart}
\usepackage[T1]{fontenc}
\usepackage[utf8]{inputenc}
\usepackage[UKenglish]{babel}
\usepackage{amsmath}
\usepackage{amssymb}
\usepackage{amsfonts}
\usepackage{amsthm}
\usepackage{enumitem}
\usepackage{mathtools}
\usepackage{csquotes}
\usepackage{hyperref}

\newtheorem{theorem}{Theorem}
\newtheorem{lemma}[theorem]{Lemma}
\newtheorem{corollary}[theorem]{Corollary}

\theoremstyle{definition}

\newcommand{\R}{\mathbb R}

\newcommand{\dd}{\mathrm d}

\newcommand{\Fr}{\textrm{Fr}}

\allowdisplaybreaks

\title{An improved upper bound for the Froude number of irrotational solitary water waves}
\author{Evgeniy Lokharu}
\address{Centre for Mathematical Sciences, Lund University, SE-22100 Lund, Sweden}
\email{evgeniy.lokharu@math.lu.se}
\author{Jörg Weber}
\address{Faculty of Mathematics, University of Vienna, AT-1090 Vienna, Austria}
\email{joerg.weber@univie.ac.at}
\date{}

\begin{document}

\begin{abstract}
A classical and central problem in the theory of water waves is to classify parameter regimes for which nontrivial solitary waves exist. In the two-dimensional, irrotational, pure gravity case, the Froude number $\Fr$ (a non-dimensional wave speed) plays the central role. So far, the best analytical result $\Fr < \sqrt{2}$ was obtained by \cite{Starr47}, while the numerical evidence of \cite{Longuet-Higgins1974-py} states $\Fr \leq 1.294$. On the other hand, as shown recently by \cite{Kozlov2023-sw}, the hypothetical upper bound $\Fr < 1.399$ is related to the existence of subharmonic bifurcations of Stokes waves. In this paper, we develop a new strategy and rigorously establish the improved upper bound $\Fr < 1.3451$, which is the first rigorous improvement of Starr's bound. In this process, we establish several new inequalities for the relative horizontal velocity, which are of separate interest and for which we delicately make use of the bound on the slope of the surface profile established by \cite{Amick1987}. As an application we show that the velocity at the bottom below the crest of any solitary wave does not exceed $47\%$ of the propagation speed.
\end{abstract}

\maketitle

\section{Introduction}

The mathematical theory of solitary waves has a long and well-known history going all the way back to the discovery of John Scott Russell in 1844. We refer to \cite{Sander1991} for an overview of early works on solitary waves. A two-dimensional solitary wave is characterized by the surface profile $\bar y = \bar \eta(\bar x-\bar c\bar t)$ that describes a fluid elevation propagating with wave speed $\bar c>0$ and localized in space, where the limits of $\bar \eta(\bar x)$ as $\bar x \to \pm \infty$ exist and are equal to the unperturbed depth level $\bar d$. (Throughout the paper, dimensional and non-dimensional variables are denoted with and without a bar, respectively; see also \eqref{eq:nondimensionalization} for the precise relations.) In a frame moving along with the wave speed $\bar c$, the solitary wave is subject to the equations
\begin{subequations}\label{eq:original}
	\begin{align}
		\bar u\bar u_{\bar x} + \bar v\bar u_{\bar y} &= -\bar p_{\bar x} & \text{in }0<\bar y<\bar \eta(\bar x),\label{eq:Euler1}\\
		\bar u\bar v_{\bar x} + \bar v\bar v_{\bar y} &= -\bar p_{\bar y} - \bar g & \text{in }0<\bar y<\bar \eta(\bar x),\label{eq:Euler2}\\
		\bar u_{\bar x} + \bar v_{\bar y} &= 0 & \text{in }0<\bar y<\bar \eta(\bar x),\label{eq:incompressible}\\
		\bar v_{\bar x} - \bar u_{\bar y} &= 0 & \text{in }0<\bar y<\bar \eta(\bar x),\label{eq:irrotational}\\
		\bar p &= 0 & \text{on }\bar y=\bar \eta(\bar x),\label{eq:dynamic}\\
		\bar v &= \bar u\bar \eta_{\bar x} & \text{on }\bar y=\bar \eta(\bar x),\label{eq:kinematic_top_velocity}\\
		\bar v &= 0& \text{on }\bar y=0,\label{eq:kinematic_bottom_velocity}
	\end{align}
\end{subequations}
with
\begin{align}
	\lim_{\bar x\to\pm\infty}\bar \eta(\bar x) &= \bar d,\label{eq:level_at_infty}\\
	\lim_{\bar x\to\pm\infty}(\bar u,\bar v)(\bar x,\bar y) &= (-\bar c,0).\nonumber
\end{align}
Here, the fluid domain is $0<\bar y<\bar \eta(\bar x)$, bounded above by the free surface $\bar y=\bar \eta(\bar x)$ and below by the rigid bed $\bar y=0$. Equations \eqref{eq:Euler1}--\eqref{eq:Euler2} are the Euler equations for an inviscid fluid of constant density (normalized to $1$) under the influence of gravity (with $\bar g$ the gravitational constant), while we assume the fluid to be incompressible \eqref{eq:incompressible} and irrotational \eqref{eq:irrotational}. (Throughout, lower indices denote partial derivatives.) Moreover, 
\[(\bar u,\bar v)\text{ is the relative velocity field,}\]
that is,
\[(\bar u+\bar c,\bar v)\text{ is the velocity field in a frame at rest.}\]
Furthermore, the dynamic boundary condition \eqref{eq:dynamic} ensures that the pressure $\bar p$ is equal to the constant atmospheric pressure (normalized to zero) at the surface, while \eqref{eq:kinematic_top_velocity} and \eqref{eq:kinematic_bottom_velocity} are the kinematic boundary conditions at the surface and the bed, respectively.

Then, following \cite{froude1874experiments}, we define the (non-dimensional) Froude number of a traveling solitary wave as
\begin{equation}\label{eq:def_Fr}
	\Fr = \frac{\bar c}{\sqrt{\bar g\bar d}}.
\end{equation}
A highly interesting question, which has attracted a lot of research, is for which values of the Froude number solitary waves can and do exist. It is well-known that solitary waves are absent in the parameter regime $\Fr \leq 1$; see \cite{Starr47}, \cite{AmickToland81b}, \cite{mcleod}, and \cite{KozlovLokharuWheeler21}. This bound is sharp since small-amplitude solitary waves have Froude number larger, but close to $1$.

Less is known for the upper bound. Let us recall \cite{Starr47}, who derived the striking identity
\[
\Fr^2 = 1+ \frac{3}{2\bar d} \frac{\int_\R (\bar \eta - \bar d)^2\,\dd \bar x}{\int_\R (\bar \eta - \bar d)\,\dd \bar x}.
\]
Starting from here and estimating the wave amplitude $\bar a$ from the Bernoulli equation, it is easy to obtain the upper bound $\Fr < 2$. Starr immediately improves this result by showing that
\begin{equation} \label{starr}
	\Fr^2 < 1 + \frac{\bar a}{\bar d},
\end{equation}
which implies
\begin{equation}\label{eq:Starr_sqrt2}
	\Fr < \sqrt{2} \approx 1.4142;
\end{equation}
see \cite{kp} for another derivation. Surprisingly, the inequality $\Fr < \sqrt{2}$ has remained the best analytical result so far, although the numerical evidence from \cite{Longuet-Higgins1974-py} (see also \cite{Miles1980}) suggests a significantly better estimate $\Fr \leq 1.294$. This numerics, however, restricts itself to waves on a bifurcation branch connecting to undisturbed still water.

While it is physically interesting to prove rigorously how fast solitary waves can travel at most (say, in the context of tsunamis), bounds for the Froude number have also turned out to be helpful to say more about the bifurcation diagram for periodic Stokes waves.
Indeed, under the hypothetical improved bound $\Fr < 1.399$, \cite{Kozlov2023-sw} was able to show the existence of subharmonic bifurcations of Stokes waves.

Thus, the goal of this paper is to improve \eqref{eq:Starr_sqrt2} and, as a byproduct, confirm the conditional result of \cite{Kozlov2023-sw}. Our following main theorem is the first rigorous improvement of Starr's bound.

\begin{theorem} \label{thm:main}
	The Froude number of any solitary wave is bounded from above by $1.3451$.
\end{theorem}

Some comments on this result:
\begin{itemize}
	\item Our result is for the full Euler model, that is, we do not restrict ourselves to any asymptotic regime such as the shallow water regime. Note also that only solitary waves with Froude number very close to $1$ are accurately modeled by a Korteweg--de Vries soliton. The regime we are mainly dealing with is far from that.
	\item Our result holds for \emph{any} solitary wave, that is, for waves not necessarily connected to undisturbed still water by some bifurcation branch, as opposed to the numerics in \cite{Longuet-Higgins1974-py}. While it is not known if such other branches exist, there is numerical evidence for their existence, at least in the case of (small) constant vorticity; see \cite{VandenBroeck_1994}.
	\item The number $1.3451$ is the root of a complicated analytic function and stated with an accuracy of $5$ significant digits. We shall always, as a convention, state numerical estimates and values with an accuracy of $5$ significant digits.
\end{itemize}

Compared to \cite{Starr47}, our argument is completely different and is not based on any integral identities. Instead, we show by a maximum principle argument that a carefully chosen harmonic function given in terms of the relative velocity has a fixed sign on one half of the fluid domain; see Lemma \ref{thm:slope}. The proof delicately hinges on the bound on the slope of the surface profile obtained by \cite{Amick1987}. Based on Lemma \ref{thm:slope}, we derive the inequality \eqref{eq:est_psiy_up_good} for the relative horizontal velocity at the crest line; see Corollary \ref{cor:u_crestline}. This inequality will then be the main ingredient to prove Theorem \ref{thm:main} in Section \ref{sec:proof}.

In Section \ref{sec:lemmas}, we establish a novel inequality for the relative horizontal velocity at the bottom, which is of separate interest; see Lemma \ref{l:lower_bottom}. Indeed, as an application of Theorem \ref{thm:main} this implies a novel estimate on the velocity at the bottom right below the crest: For any solitary wave this velocity is restricted to at most $46.376\%$ of the propagation speed.

\section{Preliminaries}\label{sec:preliminaries}
As is quite common in the theory of two-dimensional water waves, it is more convenient to work with the stream function instead of the velocity field. Here, the stream function $\bar \psi$ and the relative velocity field $(\bar u,\bar v)$ are related by
\[\bar u=-\bar \psi_{\bar y}<0,\qquad \bar v=\bar \psi_{\bar x}.\]
It turns out that for our purposes it is most convenient to work with non-dimensional variables as in \cite{KeadyNorbury75}, where the mass flux and the gravitational constant are scaled to unity. More precisely, the (dimensional) stream function is measured in units of the mass flux $\bar M$ ($=\bar c\bar d$ for a solitary wave), lengths in units of $\bar M^{2/3}\bar g^{-1/3}$, and velocities in units of $\bar M^{1/3}\bar g^{1/3}$. That is,
\begin{equation}\label{eq:nondimensionalization}
	\bar\psi=\bar M\psi,\quad (\bar x,\bar y,\bar\eta,\bar d)=\bar M^{2/3}\bar g^{-1/3}(x,y,\eta,d),\quad(\bar u,\bar v,\bar c)=\bar M^{1/3}\bar g^{1/3}(u,v,c),
\end{equation}
where the variables without bars are non-dimensional. Note that still
\begin{equation}\label{eq:def_stream_function}
	u=-\psi_y<0,\qquad v=\psi_x.
\end{equation}
Then, the non-dimensionalized version of \eqref{eq:original} reads
\begin{subequations}\label{eq:WaterWaveProblem}
	\begin{align}
		\Delta\psi &= 0 & \text{in } 0 < y < \eta(x),\\
		\frac{|\nabla\psi|^2}{2} + y &= r & \text{on } y = \eta(x),\\
		\psi &= 1 & \text{on } y = \eta(x),\\
		\psi &= 0 & \text{on } y = 0.
	\end{align}
\end{subequations}
Above, $r$ is referred to as the Bernoulli constant and is a parameter of the problem. In what follows we only require $u=-\psi_y$ to be continuous within the fluid domain up to the boundary, while the surface profile is just continuous. In particular, all our results will be valid for extreme waves.

Throughout the rest of this paper, we shall only work with the system \eqref{eq:WaterWaveProblem}, but state the results also in terms of the relative fluid velocity $(u,v)$, which is the physically more relevant quantity. In particular, the reader should always keep the relation \eqref{eq:def_stream_function} in mind.

We are concerned with solitary wave solutions of \eqref{eq:WaterWaveProblem} that can be characterized by surface profiles $\eta$ satisfying the non-dimensional counterpart of \eqref{eq:level_at_infty}, that is,
\begin{equation}\label{eq:eta_infinity}
	\lim_{x\to\pm\infty}\eta(x)=d.
\end{equation}
This means that
\begin{equation}\label{eq:psi_lim_infty}
	\psi\to\frac yd,\quad u\to-\frac1d,\quad v\to 0,\quad \textrm{for }x\to\pm\infty.
\end{equation}
In particular,
\begin{equation}\label{eq:c_nondim}
	c=\frac1d.
\end{equation}
In fact, any solitary wave is a wave of elevation with $\eta(x) > d$, while the symmetric surface profile $\eta$ is strictly monotonically decreasing for $x > 0$ (and increasing for $x<0$); for more details see \cite{CraigSternberg88} and \cite{KozlovLokharuWheeler21}. The crest is located at the vertical $x = 0$, and we denote
\[\hat\eta \coloneqq \max \eta = \eta(0).\]
Given any Bernoulli constant $r > 3/2$ one finds two conjugate streams with depths $d_-(r)$ and $d_+(r)$ that are the unique roots of the equation
\begin{equation}\label{eq:d}
	\frac{1}{2d^2} + d = r.
\end{equation}
Note that the assumption $r > 3/2$ is valid for all water waves; check \cite{AmickToland81b}, \cite{KeadyNorbury75}, and \cite{Kozlov2009-xp}. Thus, any solitary water wave is subject to the bounds
\begin{equation}\label{eq:hateta_interval}
	d = d_-(r) < 1 < d_+(r) < \hat{\eta} \leq r.
\end{equation}
See \cite{KeadyNorbury75} and \cite{Kozlov2009-xp} for more details. Finally, we note that the Froude number $\Fr$ can be computed as
\begin{equation}\label{eq:Fr_nondim}
	\Fr = d^{-3/2} > 1
\end{equation}
in view of \eqref{eq:def_Fr} and \eqref{eq:c_nondim} together with our normalization $g=1$.

Another important quantity, introduced by \cite{Benjamin1954-bl}, is the flow force constant given by
\begin{equation}\label{ff}
	S \coloneqq \int_0^\eta\left(\frac{u^2}{2}-\frac{v^2}{2}+r-y\right)\,\dd y = \frac{d^2}{2} + \frac{1}{d};
\end{equation}
here, the first expression is easily seen to be independent of $x$, just by means of \eqref{eq:WaterWaveProblem}, while the second identity follows by taking the limits $x\to\pm\infty$, using \eqref{eq:eta_infinity} and \eqref{eq:psi_lim_infty}. As for an upper bound for the Froude number, let us present the following simple argument which involves $S$ and leads to the inequality $\Fr < \sqrt{2}$, and is in fact very similar to an argument in \cite{kp}.

We apply \eqref{ff} right below the crest (i.e., for $x=0$) and use the Cauchy--Schwarz inequality (noting that $u$ is not constant)\footnote{The Cauchy--Schwarz inequality asserts that $|\int_I fg|\le(\int_I|f|^2)^{1/2}(\int_I|g|^2)^{1/2}$ for any square-integrable functions $f,g$ on an interval $I\subset\R$. Moreover, equality holds if and only if $f$ and $g$ are linearly dependent. In \eqref{est:CS_brutal}, the inequality is applied for $f\equiv1$, $g=u$, $I=(0,\hat\eta)$.}
\begin{equation}\label{est:CS_brutal}
	1=-\int_0^{\hat\eta}u\,\dd y<\sqrt{\hat\eta}\left(\int_0^{\hat\eta}u^2\,\dd y\right)^{1/2}
\end{equation}
to estimate
\begin{equation}\label{stilde}
	S = r\hat\eta - \frac{\hat\eta^2}{2} + \frac12 \int_0^{\hat\eta} u^2\,\dd y > r \hat{\eta} - \frac{\hat{\eta}^2}{2} + \frac{1}{2 \hat{\eta}} \eqqcolon \mathcal{S}(\hat{\eta}).
\end{equation}
A direct computation reveals that $\mathcal{S}'(t) = 0$ only for $t = d$ and $t = d_+$. Thus, by \eqref{eq:hateta_interval} and since $\mathcal{S}$ is decreasing for $t > d_+$, we obtain 
\begin{equation} \label{eq:intro:S}
	S > \mathcal{S}(r) = \frac{r^2}{2} + \frac{1}{2r}.
\end{equation}
In combination with the identities \eqref{eq:d} and \eqref{ff}, one immediately finds that this inequality is only true if $\Fr < \sqrt{2}$. Furthermore, one can plug \eqref{eq:d}, \eqref{ff}, and $\hat{\eta} = d + a$ into \eqref{stilde} to obtain Starr's inequality \eqref{starr}, and that is completely avoiding any integral identities. Let us also remark that the bound $\Fr<\sqrt2$ yields the upper bound $r< 4^{1/3}\approx 1.5874$ via \eqref{eq:d} and \eqref{eq:Fr_nondim}; therefore, throughout this paper it is sufficient to consider
\begin{equation}\label{range_r}
	r\in (\tfrac32, 4^{1/3}).
\end{equation}

Our proof of the main theorem is based on improving the Cauchy--Schwarz inequality \eqref{est:CS_brutal} or, equivalently, on estimating the integral
\begin{equation}\label{int:u-average}
	\frac12 \int_0^{\hat{\eta}} \left( u + \hat{\eta}^{-1} \right)^2 \, \dd y
\end{equation}
from below, where $-\hat\eta^{-1}$ is the average of $u$.

\section{Estimates below the crest}\label{sec:slope}
Since the classical works of \cite{Starr47}, or \cite{kp} and \cite{Long56}, it is known that a key ingredient is a careful estimate of $u$ at the crest line. A basic property is the following: By a maximum principle argument for $v=\psi_x$ on $x>0$, using the monotonicity of the surface profile, we have
\[\psi_{yy}(0,y) = -\psi_{xx}(0,y)<0,\text{ that is, }u_y(0,y)>0\text{ for }y>0.\]
This means that the horizontal velocity is strictly increasing from bottom to top at the crest line. Little, however, does it say how fast this increase is. We shall make this much more precise in the following completely new result.

\begin{lemma} \label{thm:slope}
	Given an arbitrary solitary wave solution, consider the harmonic function 
	\[
	f \coloneqq (\psi_y^2-\psi_x^2)\psi_{xy}-2\psi_x\psi_y\psi_{xx}+\tfrac35\psi_x = -(u^2-v^2)u_x+2u u_y v+\tfrac35v.
	\]
	Then, $f$ is negative everywhere in $\Omega\coloneqq\{(x,y)\in\R^2:0<y<\eta(x),-\infty<x<0\}$, which is the left open half of the fluid domain, where $\eta' > 0$.
\end{lemma}

We formulate the claim for solitary waves, while it is also true for any Stokes wave solution.

\begin{proof}
	First, note that $f$ is indeed harmonic, which follows from a straightforward computation. Moreover, $f < 0$ along the bottom of $\Omega$ because $\psi_{xy} < 0$ there. At the crest line $f = 0$, so we assume that there is a global positive maximum, attained  somewhere at the surface. Let us compute the normal derivative there. In the following, all quantities are understood to be evaluated at this maximum. Introducing the tangential derivative
	\[T\coloneqq \partial_x + \eta'\partial_y,\]
	we know that
	\[Tf = 0.\]
	Moreover, by \eqref{eq:WaterWaveProblem} we have
	\[\Delta\psi_x = \Delta\psi_y = 0, \quad T^2\left(\frac{|\nabla\psi|^2}{2}+y\right) = 0, \quad T^3\psi = 0.\]
	These, in total five equations form an inhomogeneous system of linear equations for the third derivatives $\psi_{xxx}$, $\psi_{xxy}$, $\psi_{xyy}$, $\psi_{yyy}$, and $\eta'''$, with lower order derivatives as coefficients. Thus, these third derivatives can be uniquely expressed in terms of lower order derivatives. Similarly, with the same line of reasoning, by
	\[\Delta\psi = 0,\quad T\left(\frac{|\nabla\psi|^2}{2}+y\right) = 0, \quad T^2\psi = 0,\quad (\psi_y^2-\psi_x^2)\psi_{xy}-2\psi_x\psi_y\psi_{xx}+\tfrac35\psi_x = f,\]
	the second order derivatives $\psi_{xx}$, $\psi_{xy}$, $\psi_{yy}$, and $\eta''$ can be expressed in terms of first order derivatives and $f$. Finally, through $T\psi=0$, $\psi_x$ can be written in terms of $\psi_y$ and $\eta'$. Altogether, we see that the normal derivative of $f$ can be written in terms of $\psi_y$, $\eta'$, and $f$. Omitting the lengthy details of this computation, we arrive at
	\begin{equation} \label{eq:slope:1}
		-\eta' f_x + f_y = A + B f + C f^2,
	\end{equation}
	where
	\begin{align*}
		A &= -\eta' \frac{3-24\eta'^2+57\eta'^4-32\eta'^6+4\eta'^8}{25(1-\eta'^2)^3\psi_y},\\
		B &= -\frac{(1+\eta'^2)^2(2-5\eta'^2-9\eta'^4+4\eta'^6)}{20\eta'^2(1-\eta'^2)^3\psi_y^2},\\
		C &= -\frac{(1+\eta'^2)^5}{16 \eta'^3 (1-\eta'^2)^3\psi_y^3}.
	\end{align*}
	Now notice
	\begin{equation}\label{eq:maxslope}
		|\eta'| < 0.60443,
	\end{equation}
	which corresponds to the upper bound $31.15^\circ$ for the slope obtained in \cite{Amick1987}. Thus, it is easy to see that $A<0$. 
	Note that the coefficient $3/5$ in the definition of $f$ is the minimal value of this coefficient such that $A<0$ for all possible values of $\eta'$.
	
	Thus, $A,C < 0$. If $B \leq 0$, then we immediately arrive at a contradiction with the Hopf lemma. Assume now $B>0$, which forces 
	\begin{equation}\label{eq:eta_low}
		\eta' > 0.52729.
	\end{equation}
	Then we complete the expression in \eqref{eq:slope:1} to a full square
	\[
	-\eta' f_x + f_y = A + D^2 + (B f + C f^2-D^2),
	\]
	so that the expression in brackets is a negative full square, and compute
	\[
	A + D^2 = \frac{4-24\eta'^2+21\eta'^4+40\eta'^6}{100 \eta'(1-\eta'^4)\psi_y}.
	\]
	This quantity is strictly negative for $0.49028 < \eta' < 0.60605$, which is satisfied in view of \eqref{eq:maxslope} and \eqref{eq:eta_low}. Thus, the normal derivative of $f$ is negative, leading to a contradiction. Therefore, $f < 0$ everywhere in $\Omega$.
\end{proof}

\begin{corollary}\label{cor:u_crestline}
	Everywhere below the crest we have
	\begin{equation}\label{eq:cor_monotone}
		(\psi_y^2 \psi_{yy}+\tfrac35 \psi_y)_y = -(u^2 u_y+\tfrac35 u)_y < 0.
	\end{equation}
	Moreover, at the crest line,
	\begin{equation}\label{eq:est_psiy_up_good}
		-u = \psi_y \le \frac{\sqrt2\sqrt{80r-27y-53\hat\eta}}{3\sqrt5} - \frac{\sqrt{2(r-\hat\eta)}}{3}.
	\end{equation}
	In particular,
	\begin{equation}\label{eq:est_psiy_up_good_extreme}
		u^2 \leq \tfrac65 (r-y)
	\end{equation}
	at the crest line for an extreme wave.
\end{corollary}
\begin{proof}
	First, \eqref{eq:cor_monotone} follows directly from the Hopf lemma, since $f = 0$ below the crest, while $0$ is the global maximum of $f$ on the closure of $\Omega$.
	
	Let us now consider an extreme wave for which $\psi_y(0,\hat\eta)=0$. Then, by \eqref{eq:cor_monotone}, $\psi_y^2 \psi_{yy}+\tfrac35 \psi_y\ge0$ at the crest line. Dividing by $\psi_y$ and integrating from $y$ to $\hat\eta$ yields \eqref{eq:est_psiy_up_good_extreme}.
	
	Consider now a non-extreme wave. At the crest, we have $\psi_{yy}=\eta''\psi_y$ in view of $T^2\psi=0$. Similarly as in the previous proof, $\eta''$ can be expressed in terms of $\psi_y$, $\eta'$, and $f$:
	\[\eta'' = -\frac{5(1+\eta'^2)^2 f + 8\eta' (1-3\eta'^2+\eta'^4)\psi_y}{20\eta' (1-\eta'^2) \psi_y^3}.\]
	Since $f/\eta'\le0$ for $x\ne0$ by Lemma \ref{thm:slope}, we get
	\[\eta'' \ge -\frac{2}{5\psi_y^2}\]
	at the crest. Therefore, at the crest line,
	\[\psi_y^2 \psi_{yy}+\tfrac35 \psi_y \ge \frac{\psi_y(0,\hat\eta)}{5} > 0.\]
	Thus, we have a differential inequality of the form
	\[h' + ah^{1/3} \ge b,\]
	with $h=\psi_y^3(0,\cdot)$, $a=9/5$, and $b=3\psi_y(0,\hat\eta)/5=3\sqrt{2(r-\hat\eta)}/5$, such that $ah^{1/3}>b$. This inequality can be rearranged and integrated from $y$ to $\hat\eta$ to yield
	\[y-\hat\eta \le \left[\frac{3(2b+ah^{1/3})h^{1/3}}{2a^2} + \frac{3b^2\ln(ah^{1/3}-b)}{a^3}\right]\Bigg|_{h=h(y)}^{h(\hat\eta)}.\]
	Therefore,
	\begin{align*}
		\psi_y^2 + \tfrac23 \sqrt{2(r-\hat\eta)}\psi_y &\le \tfrac65 (\hat\eta-y) + \tfrac29 (r-\hat\eta) \left(15 + \ln\frac{8(r-\hat\eta)}{\left(3\psi_y-\sqrt{2(r-\hat\eta)}\right)^2}\right)\\
		&\le \tfrac65 (\hat\eta-y) + \tfrac{10}{3} (r-\hat\eta),
	\end{align*}
	where in the last step we simply used $\psi_y\ge\sqrt{2(r-\hat\eta)}$. Completing the square, we arrive at \eqref{eq:est_psiy_up_good}.
\end{proof}

Now we are ready to combine the previous results, presenting a proof of the main theorem.

\section{A proof of Theorem \ref{thm:main}}\label{sec:proof}
Throughout the proof, we only estimate quantities below the crest line, that is, all quantities are understood to be evaluated at $x=0$.

Let us recall that the crucial task is to estimate the integral \eqref{int:u-average} from below. That is, we have to estimate carefully how much $u$ deviates from its average with the help of the findings of Section \ref{sec:slope}. Let us explain the main idea on how to do this: We split the integral \eqref{int:u-average} into three pieces: close to bottom, close to crest, and in between. We know that $u$ is strictly increasing with height, so $u+\hat\eta^{-1}$ is positive sufficiently close to the crest. What \enquote{sufficiently} means can be made precise with the help of our new estimate \eqref{eq:est_psiy_up_good}. Now, this inequality also provides an explicit lower bound for the \enquote{close to crest} piece. About the \enquote{in between} piece we cannot say much, unfortunately, and therefore, we simply throw it away. Finally, close to the bottom $u+\hat\eta^{-1}$ is negative, and there the integral of its square can be estimated from below by (minus) its average by means of the Cauchy--Schwarz inequality. But since the integral of $u+\hat\eta^{-1}$ along the whole crest line vanishes, that average can be translated again to an integral close to the crest, for which we can use the same techniques as before.

Put together, we then get an explicit lower bound for the integral \eqref{int:u-average}, with which $S$ can be estimated from below, similarly to \eqref{stilde}. The resulting inequality can be written only in terms of $S$, $r$, and $d$, but not of $\hat\eta$, after establishing some monotonicity property and therefore being allowed to replace $\hat\eta$ by $r$. Next, this inequality can be written only in terms of $d$ by means of \eqref{eq:d} and \eqref{ff}, then consisting only of explicit analytic expressions in $d$. It turns out that the resulting inequality can be viewed as a lower bound for $d$ (in much the same way as we recovered \eqref{eq:Starr_sqrt2} in Section \ref{sec:preliminaries}). This lower bound, being the root of a complicated, but explicit analytic function, we cannot state exactly, but rather numerically up to arbitrary precision. Finally, such a lower bound for $d$ immediately provides an upper bound for the Froude number in view of \eqref{eq:Fr_nondim}.

More concretely, the first step is to estimate
\begin{align*}
	S & = \frac{1}{2 \hat{\eta}} + r \hat{\eta} - \frac{\hat{\eta}^2}{2} + \frac12\int_0^{\hat{\eta}} (u + \hat{\eta}^{-1})^2 \, \dd y \\
	& \geq \frac{1}{2 \hat{\eta}} + r \hat{\eta} - \frac{\hat{\eta}^2}{2} + \frac12\int_0^{y_0} (u + \hat{\eta}^{-1})^2 \, \dd y + \frac12\int_{y_1}^{\hat{\eta}} (u + \hat{\eta}^{-1})^2 \, \dd y \\
	& \eqqcolon \frac{1}{2 \hat{\eta}} + r \hat{\eta} - \frac{\hat{\eta}^2}{2} + I_1 + I_2.
\end{align*}
Here, $y_0 < y_1$ are defined from the respective relations
\[
u(0,y_0) = -\hat\eta^{-1}, \quad \frac{\sqrt2\sqrt{80r-27y_1-53\hat\eta}}{3\sqrt5} - \frac{\sqrt{2(r-\hat\eta)}}{3} = \hat\eta^{-1};
\]
recall \eqref{eq:est_psiy_up_good} and that $u(0,y)$ is strictly increasing in $y$. For the first integral we use
\[\int_0^{\hat\eta} (u + \hat\eta^{-1})\, \dd y = 0\]
and the Cauchy--Schwarz inequality to find
\begin{align*}
	I_1 & \geq \frac{1}{2y_0} \left(\int_0^{y_0} |u + \hat{\eta}^{-1}| \, \dd y\right)^2 \geq \frac{1}{2y_1} \left(\int_0^{y_0} |u + \hat{\eta}^{-1}| \, \dd y\right)^2 \\
	&  = \frac{1}{2y_1} \left(\int_{y_0}^{\hat{\eta}} |u + \hat{\eta}^{-1}| \, \dd y\right)^2 \\
	&\geq  \frac{1}{2y_1} \left(\int_{y_1}^{\hat{\eta}} \left( \hat{\eta}^{-1} - \frac{\sqrt2\sqrt{80r-27y-53\hat\eta}}{3\sqrt5} + \frac{\sqrt{2(r-\hat\eta)}}{3} \right) \,\dd y\right)^2 \\
	& \eqqcolon \tilde I_1(\hat{\eta},r)
\end{align*}
with the help of \eqref{eq:est_psiy_up_good}. The second integral can be estimated directly by means of \eqref{eq:est_psiy_up_good} as
\[
I_2 \ge \frac{1}{2} \int_{y_1}^{\hat{\eta}} \left(\hat{\eta}^{-1} - \frac{\sqrt2\sqrt{80r-27y-53\hat\eta}}{3\sqrt5} + \frac{\sqrt{2(r-\hat\eta)}}{3}\right)^2\, \dd y \eqqcolon \tilde I_2(\hat{\eta},r).
\]
Both $\tilde I_1$ and $\tilde I_2$ can be computed explicitly in terms of $\hat\eta$ and $r$, but we omit the resulting complicated formulas. Combining the two results, we conclude
\[
S \geq \frac{1}{2 \hat{\eta}} + r \hat{\eta} - \frac{\hat{\eta}^2}{2} + \tilde I_1(\hat{\eta},r) + \tilde I_2(\hat{\eta},r) \eqqcolon I(\hat{\eta},r).
\]
To study $I(\hat\eta,r)$, it is beneficial to introduce $s=\sqrt{r-\hat\eta}$ and $J(s,r)=I(\hat\eta,r)-I(r,r)$ because $J$ is analytic in $(s,r)\in [0,1/\sqrt2)\times(3/2,4^{1/3})\eqqcolon D$. Here, recall \eqref{range_r} and $r-\hat\eta < r-d_+ = (2d_+^2)^{-1} < 1/2$. We find that $J\ge 0$ in $D$. We have validated this inequality rigorously with interval arithmetic up to an error of $10^{-6}$ which is sufficient to ensure the level of precision for the bounds later on. To convince the reader, we include a plot of $J(s,r)/s$ (we are dividing by $s$ since $J(0,r)=0$) in Figure \ref{fig1}.

\begin{figure}
	\centerline{\includegraphics[width=8cm]{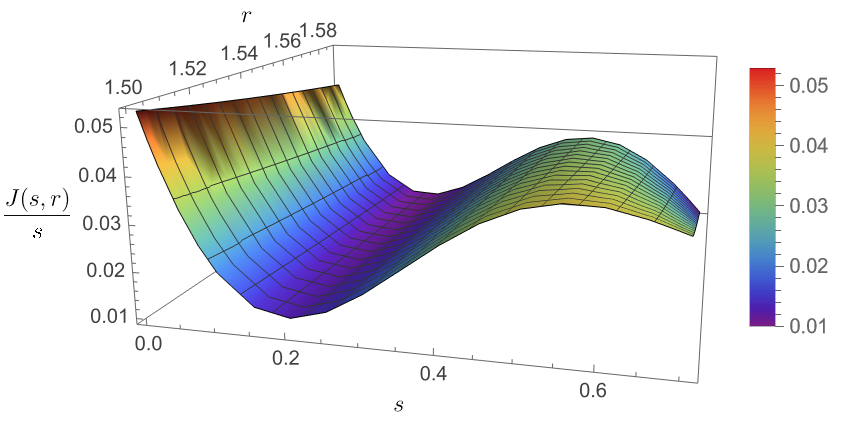}}
	\caption{Plot of $J(s,r)/s$.}
	\label{fig1}
\end{figure}

Thus, $I(\hat\eta,r)\ge I(r,r)$ and hence
\[S\ge I(r,r).\]
This, as mentioned before, can be viewed as an explicit inequality only in terms of $d$ through \eqref{eq:d} and \eqref{ff}, and we find that necessarily
\begin{equation}\label{eq:lowerbound_d}
	d > 0.82066
\end{equation}
and
\[\Fr = d^{-3/2} < 1.3451.\]
This finishes the proof of Theorem \ref{thm:main}.

\section{Further novel estimates of the horizontal velocity}\label{sec:lemmas}
While we were previously concerned with lower bounds on $u$, we shall now also derive a novel upper bound. Combined with Theorem \ref{thm:main}, this then implies a novel estimate for the velocity at the bottom below the crest. We will formulate this estimate in dimensional variables, which is most relevant for practical applications.

To this end, let us look at the bottom. A basic property is the following: Comparing with a uniform laminar flow with depth $\hat\eta$, it is clear that
\[\frac12 u^2(x,0) \ge \frac{1}{2\hat\eta^2}.\]
This inequality, however, is not very good for near-extreme waves, and we shall improve it in Lemma \ref{l:lower_bottom}.

There, we make use of the flow force function
\[
F\coloneqq\int_0^y\left(\frac{u^2}{2}-\frac{v^2}{2}+r-y'\right)\,\dd y',
\]
defined in the fluid domain and satisfying
\begin{equation}\label{fff}
	\Delta F = -1 \text{ in } 0<y<\eta(x);\quad F=S\text{ on } y=\eta(x);\quad F=0\text{ on } y=0.
\end{equation}
The following lemma follows from maximum principles applied to a carefully chosen quantity based on $F$.

\begin{lemma} \label{l:lower_bottom}
	For an arbitrary solution, we have
	\begin{equation}\label{est:psiy_bottom_better}
		\frac12 u^2(x,0) \geq \frac{S}{r} - \frac{r}{2}
	\end{equation}
	everywhere at the bottom.
\end{lemma}
\begin{proof}
	We consider the harmonic function
	\[
	h = F - \left(-\frac12 y^2 + \mu y\right), \ \ \mu = \frac{S}{\hat{\eta}} + \frac{\hat{\eta}}{2}.
	\]
	Note that $-\frac12 y^2 + \mu y$ is increasing on the interval $[0, \hat{\eta}]$, since the derivative is zero only for $y = \mu$ but 
	\begin{equation}\label{est:mu}
		\mu \ge \frac{S}{r} + \frac{r}{2} > r;
	\end{equation}
	this inequality follows from the definition of $\mu$ and its monotonicity in $\hat{\eta}$ on the interval $[0, \sqrt{2S}]$, which contains $r$ by \eqref{eq:intro:S}.
	
	Thus, we see that $h$ is non-negative at the surface by \eqref{fff} and the definition of $\mu$, so it attains its minimum everywhere at the bottom, where $h = 0$. Now, the weak maximum principle and \eqref{est:mu} give the desired inequality.
\end{proof}

Using Lemma \ref{l:lower_bottom}, we can estimate from above the velocity at the bottom right below the crest. Indeed, \eqref{est:psiy_bottom_better} expressed in terms of $d$ by means of \eqref{eq:d} and \eqref{ff} give, together with \eqref{eq:c_nondim} and our novel lower bound \eqref{eq:lowerbound_d} for $d$ obtained in the proof of Theorem \ref{thm:main},
\[
-\frac{u(0,0)}{c} \ge d\sqrt{2\left(\frac Sr-\frac r2\right)} > 0.53624.
\]
Clearly, this inequality then also holds for the ratio $-\bar u(0,0)/\bar c$ of dimensional quantities. In particular,
\[
\bar u(0,0)+\bar c < 0.46376\, \bar c.
\]
Thus, for any solitary wave the velocity at the bottom below the crest is restricted to at most $46.376\%$ of the propagation speed.

\paragraph{Acknowledgments}{This research was funded in part by the Austrian Science Fund (FWF) [10.55776/ ESP8360524]. For open access purposes, the authors have applied a CC BY public copyright license to any author-accepted manuscript version arising from this submission.}

\bibliographystyle{siam}	
\bibliography{bibliography}

\end{document}